\documentclass[a4paper,12pt,reqno]{amsart}

\usepackage{amsmath}
\usepackage{amssymb}
\usepackage{amsfonts}
\usepackage{graphicx}
\usepackage[colorlinks]{hyperref}
\renewcommand\eqref[1]{(\ref{#1})} 

\newcommand{\R}{\mathbb R}

\newcommand{\ba}{\begin{eqnarray*}}
	\newcommand{\ea}{\end{eqnarray*}}
\newcommand{\om}{\omega}
\newcommand{\be}{\begin{equation}}
	\newcommand{\ee}{\end{equation}}

\title[Deformed Hankel transform by modulus of continuity]{Asymptotic estimates for the growth of Deformed Hankel transform by modulus of continuity}

\author[V. Kumar]{Vishvesh Kumar}
\address{
	Vishvesh Kumar:
	\endgraf
	Department of Mathematics: Analysis, Logic and Discrete Mathematics
	\endgraf
	Ghent University, Krijgslaan 281, Building S8, B 9000 Ghent
	\endgraf
	Belgium
	\endgraf
	{\it E-mail address} {\rm Vishvesh.Kumar@ugent.be; vishveshmishra@gmail.com}}

\author[J. E. Restrepo]{Joel E. Restrepo}
\address{
	Joel E. Restrepo:
	\endgraf
	Department of Mathematics: Analysis, Logic and Discrete Mathematics 
	\endgraf
	Ghent University, Krijgslaan 281, Building S8, B 9000 Ghent
	\endgraf
	Belgium
	\endgraf
	{\it E-mail address} {\rm Joel.Restrepo@ugent.be; cocojoel89@yahoo.es}}

\author[M. Ruzhansky]{Michael Ruzhansky}
\address{Michael Ruzhansky:
	\endgraf
	Department of Mathematics: Analysis, Logic and Discrete Mathematics 
	\endgraf
	Ghent University, Krijgslaan 281, Building S8, B 9000 Ghent, Belgium
	\endgraf
	and
	\endgraf
	School of Mathematical Sciences
	\endgraf Queen Mary University of London 
	\endgraf
	United Kingdom
	\endgraf
	{\it E-mail address} {\rm Michael.Ruzhansky@ugent.be}}


\newtheorem{lem}{Lemma}[section]
\newtheorem{thm}{Theorem}[section]

\newtheorem{cor}{Corollary}[section]

\newtheorem{definition}{Definition}[section]

\newtheorem{rem}{Remark}[section]

\newtheorem{ex}{Example}[section]

\begin{document}
	\begin{abstract}
		We derive asymptotic estimates for the growth of the norm of the deformed Hankel transform on the deformed Hankel--Lipschitz space defined via a generalised modulus of continuity. The established results are similar in nature to the well-known Titchmarsh theorem, which provide a characterization of  the square integrable functions satisfying certain  Cauchy--Lipschitz condition in terms of an asymptotic estimate for the growth of the norm of their Fourier transform. We also give some necessary conditions in terms of the generalised modulus of continuity for the boundedness of the Dunkl transform of functions in  Dunkl-Lipschitz spaces, improving the Hausdorff-Young inequality for the Dunkl transform in this special scenario.
	\end{abstract}
	\maketitle
	\tableofcontents
	
	\section{Introduction}
	The well-known Titchmarsh theorem \cite[Theorem 85]{ti} provided a characterisation of H\"older Lipschitz classes in terms of the asymptotic estimate of the growth of the norm of the Fourier transform. This was extended by Younis in \cite{younis} using a different growth of functions to define the Lipschitz class. In fact, he used a particular modulus of continuity of the type $t^{\alpha}(\log t)^{\gamma}$ $(0<\alpha\leqslant 1, \gamma\in\mathbb{R})$ instead of $t^\alpha$ as $t\to+0$. This type of classes are frequently called Dini-Lipschitz classes \cite[p. 65]{nuevoman}. The classical results used the Fourier transform as a basic model to give an asymptotic estimate of the growth of the norm of functions belonging to the  Lipschitz class. Recently, it was proved that the aforementioned classical results can be extended to a more general family of growths called generalised modulus (moduli) of continuity \cite{DFR}, which, in particular, recover the Titchmarsh and Younis results. As a one of the simplest example of this class we can take the function  $t^{\gamma}\left(\ln\ln\frac1{t}\right)^{\lambda}$, where $\gamma\in(0,1)$ and $\lambda\in\mathbb{R}$, which gives rise to new results \cite{DFR}. For less trivial examples we can take $t^{\gamma+\frac{C}{\ln^{\lambda}\frac{1}{t}}}$. While, for a general one, we can consider $t^{f(t)}$ where $f$ satisfies the log-condition $|f(t+h)-f(t)|\leqslant C\frac{1}{|\ln |h||}.$ More examples and details on this class can be found e.g. in \cite[Chapter 2]{rafeirobook} or \cite{nata1999}.     
	
	In the last few decades, several authors have exploited the idea of writing Lipschitz (Dini-Lipschitz) classes by several generalized translation operators, which can balance the growth of different transforms with similar behaviour to the Fourier one. These operators have contributed to defining different generalised Lipschitz classes, which can be equivalently described by an asymptotic estimate of some Fourier like transform. Those works can be found in different scenarios such as  Euclidean spaces \cite{platonov-d,2008,Bray2,intro1,intro4,intro5} or manifolds \cite{daher2019,RSF, Kumar2,platonov}. 
	
	Dunkl analysis generalises the classical Fourier analysis
	on $\mathbb{R}$. It was first introduced by C. F. Dunkl in his seminal paper \cite{Dunkl} around thirty years ago. Dunkl analysis was
	further developed by several mathematicians. See for instance \cite{AAS,Dunkanker, plan, [4], dunkl-p,GIY,kumar,rosler} and the references cited therein.
	
	In this paper, we examine Dunkl-Lipschitz classes defined by the Dunkl-translation operator and whose behavior is controlled by a generalised modulus of continuity. We show that functions in the latter classes can be equivalently described by an asymptotic estimate for the growth of the norm of their Dunkl transform (two sided estimates). We  provide a characterisation of the $L^2$-Dunkl-Lipschitz space in terms of the asymptotic behaviour of the Dunkl transform given by the generalised modulus of continuity.  We complement our study by giving some necessary conditions to obtain the $L^{\nu,\alpha}$-boundedness of the Dunkl transform of a function in $L^p$-Dunkl-Lipschitz spaces. In particular, we obtain an improvement and generalisation of the Hausdorff-Young inequality for the Dunkl transform in the setting of Dunkl-Lipschitz spaces.    
 
The structure of the paper is as follows. In Section \ref{preli} we recall some basic definitions, concepts, and auxiliary results to be used throughout the paper. In Section \ref{mainresults} we present the main results: an analogue of the Titchmarsh theorem by modulus of continuity for Dunkl-Lipschitz classes involving a Dunkl-translation, and equivalently described by the growth of the norm of the Dunkl transform for functions in $L^{p}(\mathbb{R},|x|^{2\alpha-1}dx)$ $(\alpha>1/4,1<p\leqslant 2)$. In the end of the section we also provide some interesting remarks regarding the considered classes.     
	
	\section{Preliminaries}\label{preli}
	
	In this section, we collect  several definitions and facts about modulus of continuity and Dunkl analysis.
	
	\subsection{Modulus of continuity and Zygmund type classes}  
	\begin{definition}
		A non-negative function $\omega:I\rightarrow[0,+\infty)$ defined on a real interval $I$ is called almost increasing (almost decreasing) if there exists a constant $C\geqslant 1$ such that $\omega(t)\leqslant C\omega(s)$ for all $t,s\in I$ with $t\leqslant s$ $(t\geqslant s)$.
	\end{definition}
	
	\begin{definition}
		For a fixed $\delta_0\in(0,+\infty)$, the function $\omega:[0,\delta_0]\to[0,+\infty)$ is called a \text{\it modulus of continuity} if
		\begin{enumerate}
			\item\label{m1} $\omega(0)=0$ and $\omega(t)$ is continuous on $[0,\delta_0]$.
			\item\label{m2} $\omega(t)$ is almost increasing on $t\in[0,\delta_0]$.
			\item\label{m3} $\frac{\omega(t)}{t}$ is almost decreasing on $t\in[0,\delta_0]$.
		\end{enumerate}
	\end{definition}
	We endow the modulus of continuity considered here  with the following Zygmund type conditions:
	
	\begin{itemize}
		\item There exists a constant $C$ such that for all $t\in[0,\delta_0]$,
		\begin{equation}\label{cond4}
			\int_0^t\frac{\omega(x)}{x}\,\mathrm{d}x\leqslant C\omega(t).
		\end{equation}
		
		\item There exists a constant $C$ such that for all $t\in[0,\delta_0]$,
		\begin{equation}\label{cond3}
			\int_t^{\delta_0}\frac{\omega(x)}{x^2}\,\mathrm{d}x\leqslant C\frac{\omega(t)}{t}.
		\end{equation}
	\end{itemize}

	A modulus of continuity $\omega$ satisfying the Zygmund type conditions \eqref{cond4} and \eqref{cond3} belongs to the Bari--Stechkin classes \cite{nata111,nata10,nata1}. Classical examples of this class are $t^{\gamma}$, $t^{\gamma}\left(\ln\frac1{t}\right)^{\lambda}$, and $t^{\gamma}\left(\ln\ln\frac1{t}\right)^{\lambda}$, where $\gamma\in(0,1)$ and $\lambda\in\mathbb{R}$. 
	
	Since $\frac{\omega(t)}{t}$ is almost decreasing on $[0,\delta_0]$, we obtain the semi-additivity property: $\omega(t+s)\leqslant C_1\big(\omega(t)+\omega(s)\big),$ for any  $t,s\in[0,\delta_0],$ and consequently, the doubling property: $\omega(2t)\leqslant C_2 \omega(t),$ for any $t\in[0,\delta_0].$

	We note here that the modulus of continuity experienced critical behavior  near the origin. Therefore, it is a customary practice  to impose some restrictions on the weight for large values without affecting the generality of the results while dealing with estimates at infinity. Thus, sometimes in this paper, we will assume that $\omega(t)$ is bounded below on $[\delta_0,+\infty)$ by a positive number, and that $\omega^2(t)/t^{3}\in L^1\big([\delta_0,+\infty)\big)$. We will write these conditions explicitly in the statements when we need these.    
	
	\medskip Now, we give some definitions of some classes of functions which will help us to introduce the Zygmund type classes later. These definitions can be found in \cite{rafeirobook}.
	
	\begin{definition}
		Let us fix $\delta_0\in(0,+\infty)$. We denote by
		\begin{enumerate}
			\item $W,$ the class of continuous and positive functions $\omega$ on $(0,\delta_0]$ such that there exists finite or infinite $\displaystyle{\lim_{t\to0}\omega(t)}$. 
			\item $W_0,$  the subclass of almost increasing functions $\omega\in W$.
			\item $\overline{W},$  the class of functions $\omega\in W$ such that $t^{\alpha}\omega(t)\in W_0$ for some $\alpha\in\mathbb{R}$.
			\item $\underline{W},$  the class of functions $\omega\in W$ such that $\frac{\omega(t)}{t^{\beta}}$ is almost decreasing for some $\beta\in\mathbb{R}$. 
		\end{enumerate}
		A function $\omega$ belongs to the Zygmund class $\mathfrak{Z}^{\gamma}$ $(\gamma\in\mathbb{R})$, if $\omega\in\overline{W}$ and 
		\[\int_0^t \frac{\omega(s)ds}{s^{1+\gamma}}\leqslant C\frac{\omega(t)}{t^{\gamma}},\quad t\in(0,\delta_0),\]
		and to the Zygmund class $\mathfrak{Z}_{\lambda}$ $(\lambda\in\mathbb{R})$, if $\omega\in\underline{W}$ and
		\[\int_t^{\delta_0}\frac{\omega(s)ds}{s^{1+\lambda}}\leqslant C\frac{\omega(t)}{t^{\lambda}},\quad t\in(0,\delta_0).\]
	\end{definition}
	Notice that for a modulus of continuity $\omega$ satisfying the condition \eqref{cond3} it follows that $\omega$ belongs to the class $\mathfrak{Z}_1$, while, for a function $\omega$ satisfying the condition \eqref{cond4}, we deduce that $\omega$ is a member of the class $\mathfrak{Z}^0$.

	Next, we recall a useful result from \cite[Theorem 2.10]{rafeirobook}, which will play an essential role in the proof of the main result (see also \cite[Theorems 7.9, 7.10]{nata1999}). We first recall the Matuszewska--Orlicz type lower and upper indices of the modulus of continuity $\omega(t)$ involved in the next theorem.  They are given by  \cite{rafeirobook,nata1999}:
	\begin{align*}
		m(\omega)&=\sup_{0<t<1}\frac{\ln\left(\displaystyle\limsup_{\varepsilon\to0}\frac{\omega(\varepsilon t)}{\omega(\varepsilon)}\right)}{\ln t}=\lim_{t\to0}\frac{\ln\left(\displaystyle\limsup_{\varepsilon\to0}\frac{\omega(\varepsilon t)}{\omega(\varepsilon)}\right)}{\ln t},\end{align*} 
	and
	\begin{align*}
		M(\omega)&=\sup_{t>1}\frac{\ln\left(\displaystyle\limsup_{\varepsilon\to0}\frac{\omega(\varepsilon t)}{\omega(\varepsilon)}\right)}{\ln t}=\lim_{t\to+\infty}\frac{\ln\left(\displaystyle\limsup_{\varepsilon\to0}\frac{\omega(\varepsilon t)}{\omega(\varepsilon)}\right)}{\ln t}.
	\end{align*}
	The relation between the above indices and the Zygmund type classes is described in the next result.  
	\begin{thm}\cite[Theorem 2.10]{rafeirobook}\label{rafeiro}
		Let $\omega$ be a modulus of continuity and $\gamma\in\mathbb{R}$. Then $\omega\in\mathfrak{Z}^{\gamma}$ if and only if $m(\omega)>\gamma$, while $\omega\in\mathfrak{Z}_{\gamma}$ if and only if $M(\omega)<\gamma$. Moreover, we have 
		\begin{align} \label{eqm}
			m(\omega)=\sup\left\{\gamma>0:\frac{\omega(t)}{t^{\gamma}}\,\,\text{is almost increasing}\right\},
		\end{align}
		\begin{align}
			M(\omega)=\inf\left\{\gamma>0:\frac{\omega(t)}{t^{\gamma}}\,\,\text{is almost decreasing}\right\}.
		\end{align}
	\end{thm}
	Let us now define the following auxiliary weight function.
	\begin{definition} \label{W_w}
		Let $\om$ be a modulus of continuity such that $\om(t)/t\in L^1([0,\delta_0])$. Define
		\[
		W_\om(t)=\int_0^t\frac{\om(s)}{s} \,\mathrm{d}s, \quad t\in [0,\delta_0].
		\]
	\end{definition}
	The  weight $W_\om$ has the following properties. 
	\begin{lem}\cite[Proposition 3.1]{preprint}\label{oscarpro}
		Let $\om$ be a modulus of continuity such that $\om(t)/t\in L^1([0,\delta_0])$. Then the following statements hold:
		\begin{enumerate}
			\item There exists a constant $C>0$ such that
			\[
			\om(t)\leqslant C W_\om(t),\ \ \ t\in [0,\delta_0].
			\]
			\item The function $W_\om$ is a modulus of continuity on $[0,\delta_0]$.
		\end{enumerate}
	\end{lem}
	
	\subsection{Elements of the deformation theory on $\mathbb{R}$} In this section, we recall some basic notation and results of the so-called deformation theory and $\alpha$-Hankel transform on $\mathbb{R}$ \cite{deformation,khankel}.
	
	Let $L_{{p,\alpha}}:=L^{p}(\mathbb{R},|x|^{2\alpha-1})$ $(1\leqslant p <+\infty,\alpha>1/4)$ be the space of functions from $L^{p}$ defined on $\mathbb{R}$ endowed with the finite norm 
	\[
\|f\|_{p,\alpha}=\left(c_{\alpha}\int_{\mathbb{R}}|f(x)|^{p}|x|^{2\alpha-1}{\rm d}x\right)^{1/p}, \quad c_\alpha=\frac{1}{2\Gamma(2\alpha)}.
	\]
	Similarly, we can define $L_{\infty, \alpha}$ and associted norm $\|\cdot\|_{\infty, \alpha}.$
	
	The new Dunkl transform (recently called $\alpha$-Hankel transform) of order $\alpha>1/4$ for $f\in L_{1,\alpha}$ is given by 
	\[
	\mathcal{F}_{\alpha}(f)(\lambda)=c_{\alpha}\int_{\mathbb{R}}f(x)B_{\alpha}(\lambda x)|x|^{2\alpha-1}{\rm d}x,\quad \lambda\in\mathbb{R},
	\]
	where the kernel $B_{\alpha}$ is defined by 
	\[B_{\alpha}(\lambda x)=j_{2\alpha-1}(2|\lambda x|^{1/2})-\frac{\Gamma(2\alpha)}{\Gamma(2\alpha+2)}\lambda xj_{2\alpha+1}(x),\]
	where $j_{\alpha}(x)$ is the normalized Bessel function of the first kind of order $\alpha$ given by
	\[j_{\alpha}(x)=\Gamma(1+\alpha)\sum_{k=0}^{+\infty}\frac{(-1)^k}{k! \Gamma(k+\alpha+1)}\left(\frac{x}{2}\right)^{2k},\,\, x\in\mathbb{R}.\]
An important feature of the kernel $B_\alpha(x)$ is that 
\begin{equation}\label{nearzero}
 |B_\alpha(\lambda x)-1|\geqslant C|\lambda x|,\quad\text{for all sufficient small}\quad \lambda x\quad \text{and}\quad \alpha>1/4,    
\end{equation}
in view of the asymptotic behavior near $0$ of the normalized Bessel function, see e.g. \cite[Chapter VII]{watson}. Using the fact that $|B_\alpha(\lambda x)| \leqslant 1$ one can easily see that 
	\begin{equation} \label{L1esti}
		\|\mathcal{F}_{\alpha} f\|_{\infty, \alpha} \leqslant \|f\|_{1, \alpha}.
	\end{equation}
	
The kernel $B_{\alpha}(x)$ is associated with the solution \cite[Theorem 5.7]{deformation} of the following equation 
	\[|x|\Delta_\alpha B_\alpha(\lambda x)+|\lambda|B_\alpha (\lambda x),\]
	where $\Delta_\alpha$ is the well-known Dunkl Laplacian \cite{dunkl}.
 
 The $\alpha$-Hankel transform satisfies the following Plancherel identity: 
	\begin{align}\label{planc}
		\|\mathcal{F}_\alpha f\|_{2, \alpha}=\|f\|_{2, \alpha}.
	\end{align}
The inverse transform is defined as 
	\[
f(x)=c_{\alpha}\int_{\mathbb{R}}\mathcal{F}_{\alpha}(f)(\lambda)B_{\alpha}(\lambda x)|\lambda|^{2\alpha-1}{\rm d}\lambda.
	\]
	By the Plancherel identity \eqref{planc} and the estimate \eqref{L1esti} with the help of the Riesz–Thorin interpolation theorem give the following Hausdorff-Young inequality for the $\alpha$-Hankel transform \cite[Corollary 3.2]{khankel}:
	\begin{equation}\label{(1)}
		\|\mathcal{F}_{\alpha}(f)\|_{q,\alpha}\leqslant C\|f\|_{p,\alpha},
	\end{equation}
	for any $f\in L_{\alpha,p}$ with $1\leqslant p\leqslant 2$  and $q$ such that $\frac{1}{p}+\frac{1}{q}=1,$ where the positive constant $C$ does not depend on $f$.  
	
In this paper we will use the following translation operator defined by \cite[Formula (1.12)]{khankel}:
\[
T_h(f)(x)=\int_{\mathbb{R}}f(z)K_{L}(x,h,z)|z|^{2\alpha-1}{\rm d}z, \quad x,h\in[-\infty,+\infty],
\]
which  satisfies that 
 \begin{equation}\label{(2)}
		\mathcal{F}_{\alpha}(T_{h}f)(\lambda)=B_{\alpha}(\lambda h)\mathcal{F}_{\alpha}(f)(\lambda),\quad\text{for almost every} \quad \lambda\in \mathbb{R},
	\end{equation}
and any $f\in L_{p,\alpha}$ $(1\leqslant p\leqslant 2)$, $\alpha>1/2$ and $h\in\mathbb{R}$ \cite[Theorem 3.3]{khankel}. The explicit form of the kernel $K_L$ can be found in \cite[Theorem 1.1]{khankel}. 
 
It was proved in \cite[Theorem 3.3]{khankel} that $T_h$ is a bounded operator on $L_{\alpha, p}$ for $1\leqslant p\leqslant +\infty.$

	\section{Main Results}\label{mainresults}
	
	In this section, we present the Titchmarsh theorem for Dunkl type transforms by using a modulus of continuity. We present our results by giving two sided estimates with necessary and sufficient conditions. We characterise some Dunkl Lipschitz functions involving a Dunkl translation operator by an asymptotic estimate for the growth of the norm of the Dunkl transform, which in both cases uses a modulus of continuity.

\subsection{Generalized Dunkl--Lipschitz classes by modulus of continuity}
 
We first introduce the generalized Dunkl--Lipschitz classes defined using the Dunkl translation operator. 
	
	\begin{definition}
		Let $\omega$ be a modulus of continuity. A function $f\in L_{p,\alpha}$ $(1<p\leqslant 2, \alpha\geqslant -1/2)$ is said to be in the $\omega$-Dunkl Lipschitz class,  denoted by $DLip(\omega,p)$, if there exists a constant $C>0.$ such that for all sufficiently small $h>0$, 
		\[
		\|T_{h}f-f\|_{p,\alpha}\leqslant C\omega(h).
		\]
		
		We endow the space $DLip(\omega,p)$ with the  norm
		\[
\|f\|_{DLip(\omega,p)}=\|f\|_{p,\alpha}+\|f\|_{\#,DLip(\omega,p)},\quad \|f\|_{\#,DLip(\omega,p)}=\sup_{h\neq 0}\frac{\|T_{h}f-f\|_{p,\alpha}}{\omega(h)},
		\]
		so that it becomes a Banach space.
	\end{definition}
	
	Now, we are ready to present one of the main results of this paper.
	\begin{thm}\label{main1}
		Let $\omega$ be a modulus of continuity. The following statements hold.  
		\begin{enumerate}
			\item Assume that $\omega$ satisfies the Zygmund condition \eqref{cond4} and $f\in L_{p,\alpha}$ with $1<p\leqslant 2$ and $\alpha> 1/4$. If $f\in DLip(\omega,p)$ then there exists a constant $C>0$ such that, for all sufficiently small $h>0$, we have
			\[
				\int_{|\lambda|\geqslant \frac{1}{h}}|\mathcal{F}_{\alpha}(f)(\lambda)|^{q}|\lambda|^{2\alpha-1}{\rm d}\lambda\leqslant C\omega^{q}(h),\quad \frac{1}{p}+\frac{1}{q}=1.
			\]
			\item Assume that $\omega$ satisfies the Zygmund condition \eqref{cond3} and $f\in L_{2,\alpha}$ with $\alpha>1/4$. We also assume that $\omega$ is bounded below on $[\delta_0,+\infty)$ by a positive number, and that $\omega^2(t)/t^{3}\in L^1\big([\delta_0,+\infty)\big)$. If there exists a constant $C>0$ such that for all sufficiently small $h>0$,
			\begin{equation}\label{L2asymtotic}
				\int_{|\lambda|\geqslant \frac{1}{h}}|\mathcal{F}_{\alpha}(f)(\lambda)|^{2}|\lambda|^{2\alpha-1}{\rm d}\lambda\leqslant C\omega^{2}(h),
			\end{equation}
			then $f\in DLip(\omega,2)$.
		\end{enumerate}
	\end{thm}
	\begin{proof}
		We begin the proof by proving the first assertion. For that, we take $|\lambda|\in\big[\frac{1}{h},\frac{2}{h}\big]$ for $h>0$. Since $|\lambda h|\geqslant 1$ and by \cite[Lemma 1]{extension} we have $C\leqslant |1-B_{\alpha}(\lambda h)|^{q}$ for $q$ such that $1/p+1/q=1$ and $C$ is a positive constant which depends on $\alpha$ and $q$. By  \eqref{(1)}, \eqref{(2)} and $f\in DLip(\omega,p)$, there exists a constant $C$ such that 
		\begin{align*}
			\int_{\frac{1}{h}\leqslant |\lambda|\leqslant \frac{2}{h}}|\mathcal{F}_{\alpha}(f)(\lambda)|^{q}&|\lambda|^{2\alpha-1}{\rm d}\lambda \leqslant\frac{1}{C}\int_{\frac{1}{h}\leqslant |\lambda|\leqslant \frac{2}{h}}|1-B_{\alpha}(\lambda h)|^{q}|\mathcal{F}_{\alpha}(f)(\lambda)|^{q}|\lambda|^{2\alpha-1}{\rm d}\lambda \\
			& \leqslant\frac{1}{C}\int_{\mathbb{R}}|1-B_{\alpha}(\lambda h)|^{q}|\mathcal{F}_{\alpha}(f)(\lambda)|^{q}|\lambda|^{2\alpha-1}{\rm d}\lambda \\
			&\leqslant\frac{1}{C}\|\mathcal{F}_\alpha(T_{h}f-f)\|_{q,\alpha}^{q}\leqslant C_1\|T_{h}f-f\|_{p,\alpha}^{q} \leqslant C_2\omega^{q}(h),
		\end{align*} 
		where in the penultimate inequality we have used the Hausdorff-Young inequality for $\mathcal{F}_\alpha,$ that is,  $\|\mathcal{F}_\alpha f\|_{q,\alpha}\leqslant C\|f\|_{p,\alpha}$.  
		Therefore, we have
		\[
		\int_{\frac{1}{h}\leqslant |\lambda|\leqslant \frac{2}{h}}|\mathcal{F}_{\alpha}(f)(\lambda)|^{q}|\lambda|^{2\alpha-1}{\rm d}\lambda\leqslant  C_2\omega^{q}(h),
		\]
		which implies that 
		\[
		\int_{|\lambda|\geqslant\frac{1}{h}}|\mathcal{F}_{\alpha}(f)(\lambda)|^{q}|\lambda|^{2\alpha-1}{\rm d}\lambda=\sum_{k=0}^{+\infty}\int_{\frac{2^{k}}{h}\leqslant |\lambda|\leqslant \frac{2^{k+1}}{h}}|\mathcal{F}_{\alpha}(f)(\lambda)|^{q}|\lambda|^{2\alpha-1}{\rm d}\lambda \leqslant C\sum_{k=0}^{+\infty}\omega^{q}\left(\frac{h}{2^{k}}\right).
		\]
		Notice now that by the Zygmund condition \eqref{cond4} we have that $\omega\in\mathfrak{Z}^{0}$, and then by Theorem \ref{rafeiro} it follows that $m(\omega)>0.$ Moreover, we can choose and fix some $\delta>0$ such that $m(\omega)>\delta>0$ (see \eqref{eqm}) and then the function $\frac{\omega(t)}{t^{\delta}}$ is almost increasing on $[0,\delta_0]$ by Theorem \ref{rafeiro}. Since $h/2^{k}\leqslant h$ for any $k=0,1,2,\ldots,$ we get
		\[\frac{\omega\left(h/2^{k}\right)}{\left(h/2^{k}\right)^{\delta}}\leqslant C\frac{\omega(h)}{h^{\delta}}\Longrightarrow \omega^{q}\left(h/2^{k}\right)\leqslant C\frac{\omega^{q}(h)}{2^{kq\delta}}. \]    
		Hence
		\begin{align*}
			\int_{|\lambda|\geqslant\frac{1}{h}}|\mathcal{F}_{\alpha}(f)(\lambda)|^{q}|\lambda|^{2\alpha-1}{\rm d}\lambda\leqslant C\sum_{k=0}^{+\infty}\omega^{q}\left(\frac{h}{2^{k}}\right)\leqslant C\omega^{q}(h)\sum_{k=0}^{+\infty}\left(\frac{1}{2^{k}}\right)^{\delta q}\leqslant C\omega^{q}(h).  
		\end{align*}
		Now we prove the second statement. Notice that $|B_\alpha(\lambda x)|\leqslant 1$ for any $x,\lambda\in\mathbb{R}$ and $\alpha>1/4$ (see \cite[Lemma 2.9 ]{BBS20}). This inequality was first showed in \cite[Pages 18-19]{kernel1} for $\alpha\geqslant1/2.$ So, by Plancherel theorem,  \eqref{(2)}, \eqref{L2asymtotic} and \cite[Inequality (23)]{extension} we have 
		\begin{align} \label{neqd}
			\|T_{h}f&-f\|_{2,\alpha}^{2}=\|\mathcal{F}_\alpha(T_{h}f-f)\|_{2,\alpha}^{2}=c_{\alpha}\int_{\mathbb{R}}|1-B_{\alpha}(\lambda h)|^{2}|\mathcal{F}_{\alpha}(f)(\lambda)|^{2}|\lambda|^{2\alpha-1}{\rm d}\lambda \nonumber\\
			&\leqslant 4c_\alpha\int_{|\lambda|\geqslant\frac{1}{h}}|\mathcal{F}_{\alpha}(f)(\lambda)|^{2}|\lambda|^{2\alpha-1}{\rm d}\lambda+c_\alpha\int_{|\lambda|<\frac{1}{h}}|1-B_{\alpha}(\lambda h)|^{2}|\mathcal{F}_{\alpha}(f)(\lambda)|^{2}|\lambda|^{2\alpha+1}{\rm d}\lambda \nonumber\\
			&\leqslant C\omega^{2}(h)+4c_\alpha\int_{|\lambda|<\frac{1}{h}}|\lambda h|^{2}|\mathcal{F}_{\alpha}(f)(\lambda)|^{2}|\lambda|^{2\alpha+1}{\rm d}\lambda.
		\end{align}
		We consider the two possible cases over $\lambda$, i.e. $\lambda>0$ or $\lambda<0$. 
		
		{\bf Case $\lambda>0$}. Let $\phi(t)=\int_t^{+\infty}\big|\mathcal{F}_{\alpha}(f)(s)\big|^2 s^{2\alpha-1}\,\mathrm{d}s$ for any $t>0$. By condition \eqref{L2asymtotic} we have that $\phi(t)\leqslant C\omega^2(\frac{1}{t})$ for all sufficiently large $t$. Moreover, since $f$ and therefore $\mathcal{F}_{\alpha}(f)$ are in $L_{2,\alpha}$, the function $\phi(t)$ is bounded on $\mathbb{R}$. By hypothesis we know that a modulus of continuity $\omega$ is bounded below on all intervals of the form $[\delta_0,+\infty)$, which implies that we can write $\phi(t)\leqslant C\omega^2(\frac{1}{t})$ for all $t>0$. By using the latter fact and for any large enough $x>0$, we get 
		\begin{align}\label{need}
			\int_{0}^x &t^2\big|\mathcal{F}_{\alpha}(f)(t)\big|^2 t^{2\alpha-1}\,\mathrm{d}t=-\int_0^x t^2\phi^{\prime}(t)\,\mathrm{d}t=-x^2 \phi(x)+2\int_0^x t\phi(t)\,\mathrm{d}t \nonumber\\
			&\leqslant 2\int_0^x t\phi(t)\,\mathrm{d}t\leqslant C\int_0^x t\omega^{2}(1/t)\,\mathrm{d}t=C\int_{\frac1{x}}^{+\infty}\frac{\omega^2(u)}{u^3}\,\mathrm{d}u \nonumber \\
			&=C\left(\int_{\frac1{x}}^{\delta_0}\frac{\omega^2(u)}{u^3}\,\mathrm{d}u+\int_{\delta_0}^{+\infty}\frac{\omega^2(u)}{u^3}\,\mathrm{d}u\right)=:C\big(J_1+J_2\big).
		\end{align}
		The first integral can be estimated as follows: 
		\[J_1\leqslant C_1\frac{\omega(1/x)}{1/x}\int_{\frac1{x}}^{\delta_0}\frac{\omega(u)}{u^2}\,\mathrm{d}u\leqslant C_1^{\prime}\frac{\omega^2(1/x)}{(1/x)^2},\]
		since $\omega(t)/t$ is almost decreasing on $[0,\delta_0]$ and the Zygmund condition \eqref{cond3} holds. 
		
		Next, the second integral $J_2\leqslant C_2<+\infty$ since $\omega^2(t)/t^{3}\in L^{1}\big([\delta_0,\infty)\big)$. Thus, from \eqref{need}, we obtain
		\begin{equation}\label{needaqui}
			\int_{0}^x t^2\big|\mathcal{F}_{\alpha}(f)(t)\big|^2 t^{2\alpha-1}\,\mathrm{d}t \leqslant C_3\left(\frac{\omega^2(1/x)}{(1/x)^2}+C_2\right).
		\end{equation}
		Hence, we have
		\begin{align*}
			\int_{|\lambda|<\frac{1}{h}}|\lambda h|^{2}|\mathcal{F}_{\alpha}(f)(\lambda)|^{2}&|\lambda|^{2\alpha-1}{\rm d}\lambda=h^{2}\int_0^{1/h}\lambda^{2}|\mathcal{F}_{\alpha}(f)(\lambda)|^{2}\lambda^{2\alpha-1}{\rm d}\lambda \\
			&\leqslant Ch^{2}\left(\frac{\omega^2(h)}{h^2}+C_2\right)\leqslant C\omega^2(h),
		\end{align*}
		where in the last inequality above we have used that $C\omega^2(h)\geqslant h^2$, which can be easily deduced using the fact that $0<h<1$ is sufficiently small and $\omega(t)/t$ is almost decreasing on $[0,\delta_0]$.
		
		{\bf Case $\lambda<0$}. This case is easier than the last one. By using similar arguments and $\psi(t)=\int_{-\infty}^{t}\big|\mathcal{F}_{\alpha}(f)(s)\big|^2 |s|^{2\alpha-1}\,\mathrm{d}s$ for any $t<0$, we have  
		\begin{align*}
			\int_{-\lambda<\frac{1}{h}}|\lambda|^{2}&|\mathcal{F}_{\alpha}(f)(\lambda)|^{2}|\lambda|^{2\alpha-1}{\rm d}\lambda=\int_{-\frac{1}{h}}^{0}|\lambda|^{2}|\mathcal{F}_{\alpha}(f)( \lambda)|^{2}|\lambda|^{2\alpha-1}{\rm d}\lambda \\
			&=\int_0^{\frac{1}{h}}r^{2}|\mathcal{F}_{\alpha}(f)(-r)|^{2}r^{2\alpha-1}{\rm d}r=-\int_0^{\frac{1}{h}}r^{2}\psi^{\prime}(-r)dr \\
			&=\frac{\psi(-1/h)}{h^2}-2 \int_0^{\frac{1}{h}}r\psi(-r)dr\leqslant \frac{\psi(-1/h)}{h^2}\leqslant C\frac{\omega^2(h)}{h^2}.
		\end{align*}
		Then
		\begin{align*}
			\int_{|\lambda|<\frac{1}{h}}&|\lambda h|^{2}|\mathcal{F}_{\alpha}(f)(\lambda)|^{2}|\lambda|^{2\alpha-1}{\rm d}\lambda\leqslant Ch^{2}\frac{\omega^2(h)}{h^2}.
		\end{align*}
		Therefore, by the above two cases and \eqref{neqd}, we conclude that $\|T_{h}f-f\|_{2,\alpha}^{2}\leqslant C\omega^2(h)$, proving the second statement of the theorem.
	\end{proof}

 
	The following corollary follows immediately from Theorem \ref{main1} and provides a characterisation of $DLip(\omega, 2)$ in terms of the asymptotic behaviour of the Dunkl transform.
	\begin{cor}\label{equivalence}
		Let $\omega$ be a modulus of continuity satisfying the conditions \eqref{cond4} and \eqref{cond3}. Suppose that $\omega(t)$ is bounded below on $[\delta_0,+\infty)$ by a positive number, and that $\omega^2(t)/t^{3}\in L^1\big([\delta_0,+\infty)\big)$. Let us take $f\in L_{2,\alpha}$ with $\alpha>1/4$. Then the following statements are equivalent.  
		\begin{enumerate}
			\item There exists a constant $C>0$ such that for all sufficiently small $h>0$, 
			\[
			\|T_{h}f-f\|_{2,\alpha}\leqslant C\omega(h).
			\]
			\item There exists a constant $C>0$ such that for all sufficiently small $h>0$,
			\[
			\int_{|\lambda|\geqslant \frac{1}{h}}|\mathcal{F}_{\alpha}(f)(\lambda)|^{2}|\lambda|^{2\alpha+1}{\rm d}\lambda\leqslant C\omega^{2}(h).
			\]
		\end{enumerate}
	\end{cor}

The first statement of Theorem \ref{main1} provides a necessary condition in terms of the asymptotic behaviour of $L_{q,\alpha}$ norm of the Dunkl  transform  for a function to be in  Dunkl-Lipschitz space $DLip(\omega, p).$ In this following result we provide an improvement of Theorem \ref{main1} and provide conditions in terms of the modulus of continuity for the $L_{\nu, \alpha}$-boundedness, $\nu \leqslant q$ of the Dunkl transform  of a function in  $DLip(\omega, p).$  In particular, we recover the improvement of the Hausdorff-Young inequality for the Dunkl transform in the setting of Dunkl-Lipschitz spaces (see Section \ref{exapl}).

\begin{thm}\label{fourier}
Let $\omega$ be a modulus of continuity satisfying the Zygmund condition \eqref{cond4}. Let $f\in L_{p,\alpha}$ with $1<p\leqslant 2$ and $\alpha>1/4$. If $f\in DLip(\omega,p)$ then $\mathcal{F}_{\alpha}(f)$ belongs to $L_{\nu,\alpha}$ $(\nu\leqslant q, 1/p+1/q=1)$ whenever 
\begin{equation}\label{twocond}
\frac{\omega^{\nu}(t)}{t^{(2\alpha)\left(1-\frac{\nu}{q}\right)+1}}\in L^{1}[0,1]\quad\text{and}\quad \lim_{h\to0}\frac{\omega^{\nu}(h)}{h^{(2\alpha)\left(1-\frac{\nu}{q}\right)}}<+\infty.
\end{equation}
\end{thm}
\begin{proof} Suppose that $f \in DLip(\omega, p).$ 
Applying the Hausdorff-Young inequality for $\mathcal{F}_\alpha$ along with \eqref{(2)}, we have 
\[
\int_{\mathbb{R}}|1-B_{\alpha}(\lambda h)|^{q}|\mathcal{F}_{\alpha}(f)(\lambda)|^{q}|\lambda|^{2\alpha-1}{\rm d} \lambda \leqslant \|T_h f-f\|_{p, \alpha}^q \leqslant C\omega^q(h),
\]
for sufficiently small  $h>0$  and for $q \in  [2, +\infty)$ such that $1/p+1/q=1$.  Let us first observe using \eqref{nearzero} that 
\begin{align*}
\int_{|\lambda h|<1}|\lambda|^{q}|\mathcal{F}_{\alpha}(f)(\lambda)|^{q}|\lambda|^{2\alpha-1}{\rm d}\lambda&\leqslant Ch^{-q}\int_{|\lambda h|<1}|1-e_{\alpha}(\lambda h)|^{q}|\mathcal{F}_{\alpha}(f)(\lambda)|^{q}|\lambda|^{2\alpha-1}{\rm d}\lambda \\
&\leqslant Ch^{-q}\omega^q(h),    
\end{align*}
holds for $h>0$ sufficiently small. Since $\nu\leqslant q,$ we substitute $y=\frac{1}{h}$ and then apply H\"older's inequality to obtain
\begin{align} \label{eqq16}
\int_{|\lambda|<y}|\lambda|^{\nu}&|\mathcal{F}_{\alpha}(f)(\lambda)|^{\nu}|\lambda|^{2\alpha-1}{\rm d}\lambda \nonumber\\
&\leqslant\left(\int_{|\lambda|<y}|\lambda|^{q}|\mathcal{F}_{\alpha}(f)(\lambda)|^{q}|\lambda|^{2\alpha-1}{\rm d}\lambda\right)^{\nu/q}\left(\int_{|\lambda|<y}|\lambda|^{2\alpha-1}{\rm d}\lambda\right)^{1-\nu/q} \nonumber \\
&\leqslant Cy^{\nu}\omega^{\nu}(1/y)y^{(2\alpha)\left(1-\frac{\nu}{q}\right)}    \end{align}
for sufficiently  large $y \in (1, \infty).$
Now, we have 
\begin{align} \label{eq16}
    \int_{\mathbb{R}}|\mathcal{F}_\alpha(f)(\lambda)|^{\nu}|\lambda|^{2\alpha-1}{\rm d}\lambda &= \int_{|\lambda|\leqslant 1 }|\mathcal{F}_\alpha(f)(\lambda)|^{\nu}|\lambda|^{2\alpha-1}{\rm d}\lambda+ \int_{|\lambda|>1}|\mathcal{F}_\alpha(f)(\lambda)|^{\nu}|\lambda|^{2\alpha-1}{\rm d}\lambda\nonumber \\&:=I_1+I_2.
\end{align}
Since one can easily see that 
$$I_1=\int_{|\lambda|\leqslant 1 }|\mathcal{F}_\alpha(f)(\lambda)|^{\nu}|\lambda|^{2\alpha-1}{\rm d}\lambda = O(1),$$
it is evident from \eqref{eq16} that it is enough to prove  
\begin{align}
    I_2=\lim_{y \rightarrow \infty} I(y)=\lim_{y \rightarrow \infty} \int_{1<|\lambda|<y}|\mathcal{F}_\alpha(f)(\lambda)|^{\nu}|\lambda|^{2\alpha-1}{\rm d}\lambda =O(1),
\end{align}
in order to show that
 $\mathcal{F}_{\alpha}(f)\in L_{\nu,\alpha}$.
 Let us only consider the positive part since the other one is analogous, that is, we show that
 
 $$\lim_{y \rightarrow \infty} I_+(y):=\lim_{y \rightarrow \infty} \int_{1<\lambda<y}|\mathcal{F}_\alpha(f)(\lambda)|^{\nu}\lambda^{2\alpha-1}{\rm d}\lambda =O(1).$$
 For this, we set
\[
\phi(y):=\int_1^{y}s^{\nu}|\mathcal{F}_\alpha(s)|^{\nu}s^{2\alpha-1}{\rm d}s. 
\]
Therefore, by estimate \eqref{eqq16} we obtain
\begin{align*}
I_+(y)&=\int_1^{y}s^{-\nu}\phi^{\prime}(s){\rm d}s=s^{-\nu}\phi(s)\big|_{1}^{y}\big.+\nu\int_1^{y}s^{-\nu-1}\phi(s){\rm d}s \\
&\leqslant C\omega^{\nu}(1/y)y^{(2\alpha)\left(1-\frac{\nu}{q}\right)}+\nu\int_{1}^{y}s^{-\nu-1}s^{\nu}\omega^{\nu}(1/s)s^{(2\alpha)\left(1-\frac{\nu}{q}\right)}{\rm d}s \\
&=C\omega^{\nu}(1/y)y^{(2\alpha)\left(1-\frac{\nu}{q}\right)}+\nu\int_{1}^{y}\omega^{\nu}(1/s)s^{(2\alpha)\left(1-\frac{\nu}{q}\right)-1}{\rm d}s \\
&=C\omega^{\nu}(1/y)y^{(2\alpha)\left(1-\frac{\nu}{q}\right)}+\nu\int_{1/y}^{1}\frac{\omega^{\nu}(u)}{u^{(2\alpha)\left(1-\frac{\nu}{q}\right)+1}}{\rm d}u.
\end{align*}
As we are interested to know the behaviour near infinity of $y$, let us denote $h=1/y$ to get
\[
\lim_{y \rightarrow \infty} I_+(y):=\lim_{h \rightarrow 0}I(1/h)\leqslant C\frac{\omega^{\nu}(h)}{h^{(2\alpha)\left(1-\frac{\nu}{q}\right)}}+\nu\int_{h}^{1}\frac{\omega^{\nu}(u)}{u^{(2\alpha)\left(1-\frac{\nu}{q}\right)+1}}{\rm d}u=O(1),
\]
provided that  conditions in \eqref{twocond} hold. This concludes the proof of this theorem.
\end{proof}

\subsection{Generalized Dunkl--Lipschitz classes by a larger class of modulus of continuity} 

Now we present some more general results involving a larger class of modulus of continuity $W_\omega$. In particular, the next result is a generalised version of Theorem \ref{main1}. 
	
	\begin{thm}\label{main2}
		Let $\omega$ be a modulus of continuity such that $\omega(t)/t\in L^1([0,\delta_0])$. The following statements hold. 
		
		\begin{enumerate}
			\item Let $W_\omega$ satisfy the Zygmund condition \eqref{cond4} and $f\in L_{p,\alpha}$ with $1<p\leqslant 2$ and $\alpha>1/2$. If $f\in DLip(W_\omega,p)$ then there exists a constant $C>0$ such that, for all sufficiently small $h>0$, we have
			\begin{equation}
				\int_{|\lambda|\geqslant \frac{1}{h}}|\mathcal{F}_{\alpha}(f)(\lambda)|^{q}|\lambda|^{2\alpha-1}{\rm d}\lambda\leqslant CW_\omega^{q}(h),\quad \frac{1}{p}+\frac{1}{q}=1.
			\end{equation}
			\item Let $W_\omega$ satisfy the Zygmund condition \eqref{cond3} and $f\in L_{2,\alpha}$ with $\alpha>1/2$. We also assume that $W_\omega(t)$ is bounded below on $[\delta_0,+\infty)$ by a positive number, and that $W_\omega^2(t)/t^{3}\in L^1\big([\delta_0,+\infty)\big)$. If there exists a constant $C>0$ such that, for all sufficiently small $h>0$, we have
			\begin{equation}\label{L2asymtotic-W}
				\int_{|\lambda|\geqslant \frac{1}{h}}|\mathcal{F}_{\alpha}(f)(\lambda)|^{2}|\lambda|^{2\alpha-1}{\rm d}\lambda\leqslant CW_\omega^{2}(h),
			\end{equation}
			then $f\in DLip(W_\omega,2)$.
		\end{enumerate}
	\end{thm}
	
	\proof 
	We mainly follow the same lines of the proof of the first part of Theorem \ref{main1}. One can convince that  
	\begin{align*}
		\int_{\frac{1}{h}\leqslant |\lambda|\leqslant \frac{2}{h}}|\mathcal{F}_{\alpha}(f)(\lambda)|^{q}|\lambda|^{2\alpha-1}{\rm d}\lambda \leqslant CW_\omega^{q}(h).
	\end{align*}
	
	Since $W_\omega$ is a modulus of continuity on $[0,\delta_0]$ satisfying \eqref{cond4} by Lemma \ref{oscarpro}, we can, by Theorem \ref{rafeiro}, fix some $\delta$ with $m\big(W_\omega\big)>\delta>0$ such that $\frac{W_\omega(t)}{t^{\delta}}$ is almost increasing, and deduce that
	\begin{align*}
		\int_{|\lambda|\geqslant\frac{1}{h}}&|\mathcal{F}_{\alpha}(f)(\lambda)|^{q}|\lambda|^{2\alpha-1}{\rm d}\lambda=\sum_{k=0}^{+\infty}\int_{\frac{2^{k}}{h}\leqslant |\lambda|\leqslant \frac{2^{k+1}}{h}}|\mathcal{F}_{\alpha}(f)(\lambda)|^{q}|\lambda|^{2\alpha-1}{\rm d}\lambda \\
		&\leqslant C_1 \sum_{k=0}^{+\infty}W_\omega^{q}\left(\frac{h}{2^{k}}\right)\leqslant C_2\frac{W_\omega^{q}(h)}{h^{\delta q}}\sum_{k=0}^{+\infty}\left(\frac{h}{2^{k}}\right)^{\delta q}\leqslant CW_\omega^{q}(h). 
	\end{align*}
	The rest of the proof follows the same steps of the second part of the proof of Theorem \ref{main1} by using the new modulus of continuity $W_\omega$ instead of $\omega$.  
	\endproof
	
	\begin{cor}
		Let $\omega$ be a modulus of continuity such that $\omega(t)/t\in L^1([0,\delta_0])$ and $f\in L_{2,\alpha}$ with $\alpha>1/4$. We assume that $W_\omega$ satisfies the conditions \eqref{cond4} and \eqref{cond3}.  Suppose that $W_\omega(t)$ is bounded below on $[\delta_0,+\infty)$ by a positive number, and that $W_\omega^2(t)/t^{3}\in L^1\big([\delta_0,+\infty)\big)$. Then the following statements are equivalent.  
		\begin{enumerate}
			\item There exists a constant $C>0$ such that for all sufficiently small $h>0$, 
			\[
			\|T_{h}f-f\|_{2,\alpha}\leqslant CW_\omega(h).
			\]
			\item There exists a constant $C>0$ such that for all sufficiently small $h>0$,
			\[
			\int_{|\lambda|\geqslant \frac{1}{h}}|\mathcal{F}_{\alpha}(f)(\lambda)|^{2}|\lambda|^{2\alpha-1}{\rm d}\lambda\leqslant CW_\omega^{2}(h).
			\]
		\end{enumerate}
	\end{cor}

\begin{thm}
Let $\omega$ be a modulus of continuity such that $\omega(t)/t\in L^1([0,\delta_0])$. Let $W_\omega$ satisfy the Zygmund condition \eqref{cond4}, and let $f\in L_{p,\alpha}$ with $1<p\leqslant 2$ and $\alpha>1/4$. If $f\in DLip(W_\omega,p)$ then $\mathcal{F}_{\alpha}(f)$ belongs to $L_{\nu,\alpha}$ $(\nu\leqslant q, 1/p+1/q=1)$ whenever 
\[
\frac{W_\omega^{\nu}(t)}{t^{(2\alpha)\left(1-\frac{\nu}{q}\right)+1}}\in L^{1}[0,1]\quad\text{and}\quad \lim_{h\to0}\frac{W_\omega^{\nu}(h)}{h^{(2\alpha)\left(1-\frac{\nu}{q}\right)}}<+\infty.
\]
\end{thm}
 
	\begin{rem}
		It is worth noting that Theorem \ref{main2} under the additional condition \eqref{cond4} for the modulus of continuity $\omega$ coincides with Theorem \ref{main1}, since $W_{\om}(h)\sim\omega(h)$  by \eqref{cond4} and Lemma \ref{oscarpro}. We would also like to make it clear that the classes of functions satisfying Theorem \ref{main1} are strictly contained in Theorem \ref{main2}. For example, take the modulus of continuity $\omega(t)=\ln^{-\beta}\frac{e}{t}$, $\beta>1$, which does not satisfy \eqref{cond4}. So, in general, we have
		\[DLip(\omega,p)\subset DLip(W_\omega,p).\]
		If the modulus of continuity $\omega$ satisfies \eqref{cond4}, then both classes coincide:
		\[DLip(\omega,p)=DLip(W_\omega,p).\]
	\end{rem}	

\begin{rem}
In a very particular case, we can recover some results from \cite{extension}, where the authors obtained an analog of Younis's theorem \cite[Theorem 5.2]{younis} for the Dunkl transform over a space of functions satisfying a Lipschitz type condition in $L_{p,\alpha}$ $(\alpha>1/4,1<p\leqslant 2)$ involving the Dunkl-translation operator. 
\end{rem}

\section{Illustrative examples} \label{exapl}
The aim of this section is to illustrate the power of the results discussed in the previous section in a very general setting. Sometimes, it is not clear whether there is an easy way to check if a special modulus of continuity $\omega$ satisfies either condition \eqref{cond4} or \eqref{cond3}, or both. Therefore, we present some important classes of examples where our general results are applicable. It is known that a modulus of continuity $\omega$ satisfies $0\leqslant m(\omega)\leqslant M(\omega)\leqslant +\infty$ (see \cite[P. 32]{rafeirobook}). Notice that, for a modulus of continuity to satisfy the condition \eqref{cond4} or \eqref{cond3}, it is equivalent to prove $0<m(\omega)$ and $M(\omega)<1$ respectively, see Theorem \ref{rafeiro}.       

\begin{ex}
For the classical case $\omega(t)=t^{\gamma}$ with $0<\gamma<1$, it can be easily (directly) verified that $\omega$ satisfies \eqref{cond4} or \eqref{cond3}. Therefore, all the statements in Theorem \ref{main1}, Corollary \ref{equivalence} and Theorem \ref{fourier} hold.  Therefore,  we get some similar results (in nature) given by Titchmarsh \cite[Theorems 84 and 85]{ti} by means of the $\alpha$-Hankel transform. Moreover, it follows from Theorem \ref{fourier}, that for any $f\in DLip(\omega,p)$ with $1<p\leqslant 2$ and $\alpha>1/4,$ we have  $\mathcal{F}_{\alpha}(f) \in L_{\nu,\alpha}$  whenever $\frac{(2\alpha)p}{p\gamma+(2\alpha)(p-1)} \leqslant \nu \leqslant q$, where $q$ is the Lebesgue conjugate of $p,$ that is, $ 1/p+1/q=1.$  This is an improvement of the Hausdorff-Young inequality for the class of functions belonging to  $DLip(\omega,p).$
\end{ex}

\begin{ex}
Let us now present a more interesting case. Consider the weight $\omega(t)=t^{\gamma}\left(\ln\frac{1}{t}\right)^{\theta}$ $(0<\gamma<1,\theta\in\mathbb{R})$. Here we first apply the explicit formulas of $m(\omega)$ and $M(\omega)$ given in Section \ref{preli} for this special case case. In fact, we have 
\[
m(\omega)=\lim_{t\to0}\frac{\ln\left(\displaystyle\limsup_{\varepsilon\to0}\frac{(\varepsilon t)^{\gamma}\left(\ln\frac{1}{\varepsilon t}\right)^{\theta}}{\varepsilon^{\gamma}\left(\ln\frac{1}{\varepsilon}\right)^{\theta}}\right)}{\ln t}=\lim_{t\to0}\frac{\ln t^{\gamma}}{\ln t}=\gamma.
\]
Similarly, we get
\[
M(\omega)=\lim_{t\to+\infty}\frac{\ln\left(\displaystyle\limsup_{\varepsilon\to0}\frac{(\varepsilon t)^{\gamma}\left(\ln\frac{1}{\varepsilon t}\right)^{\theta}}{\varepsilon^{\gamma}\left(\ln\frac{1}{\varepsilon}\right)^{\theta}}\right)}{\ln t}=\lim_{t\to+\infty}\frac{\ln t^{\gamma}}{\ln t}=\gamma.
\]
Since $0<\gamma<1$, we see that $\omega$ satisfies the conditions \eqref{cond4} and \eqref{cond3}. So, for this class of modulus of continuity, Theorem \ref{main1} and Corollary \ref{equivalence} hold. 
\end{ex}
 
\section{Acknowledgements}
	The authors were supported  by the FWO Odysseus 1 grant G.0H94.18N: Analysis and Partial
	Differential Equations, the Methusalem programme of the Ghent University Special Research Fund (BOF) (Grant number 01M01021) and by FWO Senior Research Grant G011522N. MR is also supported by EPSRC grant EP/R003025/2.
	

\end{document}